 \documentclass[11pt]{article}

\usepackage{palatino}
\usepackage[mathcal]{euler}

\usepackage{amsmath,amsthm,amsfonts,amssymb,latexsym,amscd,enumerate,xypic}

\theoremstyle{plain}
    \newtheorem{theorem}{Theorem}[section]
    \newtheorem{lemma}[theorem]{Lemma}
    \newtheorem{corollary}[theorem]{Corollary}
    \newtheorem{proposition}[theorem]{Proposition}

 \theoremstyle{definition}
    \newtheorem{definition}[theorem]{Definition}
    
    \newtheorem{example}[theorem]{Example}

    \newtheorem{remark}[theorem]{Remark}

\theoremstyle{remark}

\numberwithin{equation}{section}

\DeclareMathOperator{\Ad}{Ad}

\DeclareMathOperator{\ind}{index}

\DeclareMathOperator{\End}{End}
\DeclareMathOperator{\Hom}{Hom}

\DeclareMathOperator{\reg}{reg}

\DeclareMathOperator{\rank}{rank}

\DeclareMathOperator{\Spin}{Spin}

\DeclareMathOperator{\SO}{SO}

\DeclareMathOperator{\U}{U}

\DeclareMathOperator{\Cl}{Cl}

 \DeclareMathOperator{\Zeroes}{Zeroes}

    \DeclareMathOperator{\MZ}{MZ}

      \DeclareMathOperator{\DInd}{D-Ind}

\begin{document}

\newcommand{\Spinc}{\Spin^c}

    \newcommand{\R}{\mathbb{R}}
    \newcommand{\C}{\mathbb{C}} 
    \newcommand{\N}{\mathbb{N}}
    \newcommand{\Z}{\mathbb{Z}} 
    \newcommand{\Q}{\mathbb{Q}}
    \newcommand{\bK}{\mathbb{K}}

\newcommand{\g}{\mathfrak{g}}
\newcommand{\h}{\mathfrak{h}}
\newcommand{\p}{\mathfrak{p}}
\newcommand{\kg}{\mathfrak{g}} 
\newcommand{\kt}{\mathfrak{t}}
\newcommand{\kA}{\mathfrak{A}}
\newcommand{\XX}{\mathfrak{X}}
\newcommand{\kh}{\mathfrak{h}} 
\newcommand{\kp}{\mathfrak{p}}
\newcommand{\kk}{\mathfrak{k}}

\newcommand{\cE}{\mathcal{E}}
\newcommand{\cA}{\mathcal{A}}
\newcommand{\calL}{\mathcal{L}}
\newcommand{\calH}{\mathcal{H}}
\newcommand{\cO}{\mathcal{O}}
\newcommand{\cB}{\mathcal{B}}
\newcommand{\cK}{\mathcal{K}}
\newcommand{\cP}{\mathcal{P}}
\newcommand{\calD}{\mathcal{D}}
\newcommand{\cF}{\mathcal{F}}
\newcommand{\cX}{\mathcal{X}}
\newcommand{\cM}{\mathcal{M}}
\newcommand{\cS}{\mathcal{S}}
\newcommand{\cU}{\mathcal{U}}

\newcommand{\Sj}{ \sum_{j = 1}^{\dim G}}
\newcommand{\Sk}{ \sum_{k = 1}^{\dim M}}
\newcommand{\ii}{\sqrt{-1}}

\newcommand{\ddt}{\left. \frac{d}{dt}\right|_{t=0}}

\newcommand{\cSM}{\cS}
\newcommand{\PM}{P}
\newcommand{\DM}{D}
\newcommand{\LM}{L}
\newcommand{\vM}{v}

\newcommand{\Wedge}{\Lambda}

\newcommand{\specialin}{\hspace{-1mm} \in \hspace{1mm} }

\newcommand{\beq}[1]{\begin{equation} \label{#1}}
\newcommand{\eeq}{\end{equation}}
\newcommand{\bspl}{\[ \begin{split}}
\newcommand{\espl}{\end{split} \]}

\newcommand{\Utilde}{\widetilde{U}}
\newcommand{\Btilde}{\widetilde{B}}
\newcommand{\Dtilde}{\widetilde{D}}
\newcommand{\Etilde}{\widetilde{\cE}}

\newcommand{\Rhat}{\widehat{R}}

\title{An equivariant index for proper actions II: properties and applications}

\author{Peter Hochs\footnote{University of Adelaide, \texttt{peter.hochs@adelaide.edu.au}} \hspace{1mm}and Yanli Song\footnote{University of Toronto, \texttt{songyanl@math.utoronto.ca}}}

\date{\today}

\maketitle

\begin{abstract}
In the first part of this series, we defined an equivariant index without assuming the group acting or the orbit space of the action to be compact. This allowed us to generalise an index of deformed Dirac operators, defined for compact groups by Braverman. 
In this paper, we investigate properties and applications of this index.
We prove that it has an induction property that can be used to deduce various other properties of the index.
In the case of compact orbit spaces, we show how it is related to the analytic assembly map from the Baum--Connes conjecture, 
 and an index used by Mathai and Zhang. We use the index to define a notion of $K$-homological Dirac induction, and show that, under conditions, it satisfies the quantisation commutes with reduction principle.
\end{abstract}

\tableofcontents


\section{Introduction}

In part I of this series, an equivariant index was developed that applies to actions by possibly noncompact groups, and with possibly noncompact orbit spaces. To recall the definition of this index, we let $G$ be an almost connected Lie group (i.e.\ having finitely many connected components), acting properly and isometrically on a Riemannian manifold $M$. Let $\cE = \cE^+ \oplus \cE^-$ be a $\Z_2$-graded, Hermitian, $G$-equivariant vector bundle. Let $D$ be an odd, self-adjoint, $G$-equivariant, first order differential operator on $\cE$. If $M$ and $G$ are compact, then we have the usual equivariant index
\[
\ind_G(D) := [\ker D^+] - [\ker D^-] \quad \in R(G),
\]
where $D^{\pm}$ is the restriction of $D$ to $\Gamma^{\infty}(\cE^{\pm})$, and $R(G)$ is the representation ring of $G$.

The definition of the equivariant index has been generalised to cases where either the orbit space $M/G$ or the group $G$ is compact. If $M/G$ is compact, then one has the analytic assembly map from the Baum--Connes conjecture \cite{Connes94} for elliptic operators, and Kasparov's index of transversally elliptic operators \cite{Kasparov14}. If $G$ is compact, then Braverman \cite{Braverman02} defined an index of a natural class of deformed Dirac operators. This, and equivalent indices, has been used very successfully in geometric quantisation \cite{HochsSong15, Zhang14, Paradan11}.

The techniques used in the cases where $M/G$ or $G$ are compact are very different. If $M/G$ is compact, one can use operator algebraic techniques to obtain an index in the $K$-theory or $K$-homology of a $C^*$-algebra related to the group $G$. If $G$ is compact, it is natural to define an index in the completed representation ring
\[
\hat R(G) := \Bigl\{ \bigoplus_{\pi \in \hat G} m_{\pi} \pi; m_{\pi} \in \Z \Bigr\}.
\]

Because of the different approaches in the two cases, it is not immediately clear how to construct a common generalisation, i.e.\ an equivariant index that can be used when \emph{both} $M/G$ and $G$ are noncompact. This was done in \cite{HochsSong16a}, where the condition of $D$ being \emph{$G$-Fredholm} was introduced, which implies that $D$ has a \emph{$G$-index} 
\[
\ind_G(D) \in KK(C_0(G/K)\rtimes G, \C).
\]
Here $K<G$ is a maximal compact subgroup.
This $K$-homology group of the crossed product $C_0(G/K)\rtimes G$ can be identified with $\hat R(K)$ via the Morita equivalence $C_0(G/K)\rtimes G \sim C^*K$. The main result in \cite{HochsSong16a} is  that a natural class of deformed Dirac operators is $G$-Fredholm. This completes Table \ref{table intro}, by filling in the bottom-right entry.
\begin{table} \label{table intro}
\begin{tabular}{|c|c|c|}
\hline 
 & \begin{tabular}{l}$M/G$ compact,\\ $D$ transversally elliptic \end{tabular} & \begin{tabular}{l}$M/G$ noncompact, \\ $D$ a deformed Dirac operator \end{tabular}\\
\hline
$G$ compact & Atiyah, 1974 \cite{Atiyah74} & Braverman, 2002 \cite{Braverman02}\\
\hline
$G$ noncompact & Kasparov, 2015 \cite{Kasparov14} & Part I, 2016 \cite{HochsSong16a} \\
\hline
\end{tabular}
\caption{Special cases of the $G$-index}
\end{table}
As far as the authors are aware, the $G$-index is the first equivariant index that applies in cases where both $M/G$ and $G$ are noncompact. Here by an equivariant index, we mean an index taking values in an object defined purely in terms of the group acting. 

In the present paper, we study properties of the $G$-index of deformed Dirac operators. We start by proving an induction property of the index. This is an explicit description of the image of the $G$-index in $\hat R(K)$ in terms of data on a global, $K$-invariant slice $N\subset M$ such that $M = G\times_K N$. Besides giving a better understanding of the $G$-index, the induction property can also be used to prove various other properties of  it.

One such property is a relation with the analytic assembly map from the Baum--Connes conjecture \cite{Connes94} if $M/G$ is compact. If $G$ is semisimple with discrete series representations, then it turns out that the assembly map can be recovered directly from the $G$-index. Another application of the induction result is a \emph{quantisation commutes with reduction} property of the $G$-index. This generalises the main result in \cite{HochsSong15}, where $G$ was assumed to be compact.

Independently of the induction result, we give a second relation between the $G$-index and the assembly map, and an index defined by Mathai and Zhang \cite{Mathai10}, if $M/G$ is compact. But our main interest is in cases where both $M/G$ and $G$ are noncompact, which are furthest removed from existing index theory. One example of this setting is the action by $G$ on $T^*(G/K)$. In that case, we use the $G$-index to define a version of the Dirac induction isomorphism from the Connes--Kasparov conjecture, which is now defined on $\hat R(K)$.

\subsection*{Acknowledgements}

The authors thank Maxim Braverman, Nigel Higson, Gennadi Kasparov, 
 Mathai Varghese and Guoliang Yu for interesting and helpful discussions. The first author was supported by the European Union, through Marie Curie fellowship PIOF-GA-2011-299300.

\section{Preliminaries}

Let $G$ be an almost connected Lie group, i.e.\ having finitely many connected components. We recall the definition of $G$-Fredholm operators and their $G$-indices, as introduced in part I of this series \cite{HochsSong16a}. We also state the main result from \cite{HochsSong16a}, that deformed Dirac operators are $G$-Fredholm. Finally, we recall the notion of Dirac induction used in the Connes--Kasparov conjecture, which will be used in some of the applications.

\subsection{The $G$-index} \label{sec def index}

For the rest of this paper, we fix a maximal compact subgroup $K<G$, and a proper, isometric action by $G$
 on a  Riemannian manifold $M$. Let $\cE = \cE^+ \oplus \cE^- \to M$ be a $\Z_2$-graded, Hermitian vector bundle. Suppose the action by $G$ lifts to $\cE$, preserving the grading and the Hermitian metric.

Since the action is proper, Abels' theorem \cite{Abels} guarantees the existence of a smooth, equivariant map
\[
p\colon M\to G/K.
\]
Equivalently, we have a $G$-equivariant diffeomorphism
\beq{eq M G K}
M\cong G\times_K N
\eeq
via the action map, with $N := p^{-1}(eK)$.
Pullback along $p$ defines a $G$-equivarant map
\[
p^*\colon C_0(G/K)\to C_b(M).
\]
This induces a $*$-homomorphism
\[
p^*_G\colon C_0(G/K)\rtimes G \to C_b(M) \rtimes G
\]
between crossed-product $C^*$-algebras \cite{Williams07}. 

The representation by $C_b(M)$ in $L^2(\cE)$ by pointwise multiplication, and the unitary representation of $G$ in $L^2(\cE)$ combine to a representation
\[
\pi_{G, C_b(M)}\colon C_b(M)\rtimes G\to \cB(L^2(\cE)).
\]
Conider the representation
\[
\pi^p_{G, G/K}:= \pi_{G, C_b(M)} \circ p^*_G\colon C_0(G/K)\rtimes G\to \cB(L^2(\cE)).
\]
It is explicitly given by
\[
\bigl(\pi_{G, G/K}^p (\varphi)s\bigr)(m) = \int_G \varphi(g, p(m)) g\cdot (s(g^{-1}m))\, dg,
\]
for $\varphi \in C_c(G, C_0(G/K))$, $s \in L^2(\cE)$ and $m\in M$.

Let $D$ be an odd, self-adjoint, $G$-equivariant, first order differential operator on $\cE$. In \cite{HochsSong16a}, the operator $D$ was defined to be \emph{$G$-Fredholm for $p$} if the triple
\beq{eq index cycle}
\Bigl( L^2(\cE), \frac{D}{\sqrt{D^2+1}}, \pi^p_{G, G/K} \Bigr)
\eeq
is a Kasparov $(C_0(G/K)\rtimes G, \C)$ module. Then the \emph{$G$-index} of $D$ is defined as the class
\[
\ind^p_G(D) \in KK(C_0(G/K)\rtimes G, \C)
\]
of \eqref{eq index cycle}. Via the Morita equivalence $C_0(G/K) \rtimes G \sim C^*K$, this index can be identified with an element of
\[
KK(C^*K, \C) = \hat R(K) := \Bigl\{ \bigoplus_{\pi \in \hat K} m_{\pi} \pi; m_{\pi} \in \Z \Bigr\}.
\]
If $D$ is $G$-Fredholm for any map $p$ as above, then it is called \emph{$G$-Fredholm}. Its $G$-index is then independent of $p$ (see Lemma 3.2 in \cite{HochsSong16a}), and denoted by $\ind_G(D)$.

If $D$ is transversally elliptic in the sense of Definition 6.1 in \cite{Kasparov14}, and $M/G$ is compact, then $D$ is $G$-Fredholm by Proposition 6.4 and Remark 8.19 in \cite{Kasparov14}. The $G$-index of $D$ is then a generalisation of Atiyah's index of elliptic operators \cite{Atiyah74} to noncompact groups. 

Another class of $G$-Fredholm operators is obtained by applying a natural deformation to Dirac-type operators.

\subsection{Deformed Dirac operators} \label{sec def Dirac}

Suppose that $M$ is complete in the given Riemannian metric. Let  
\[
c\colon TM \to \End(\cE)
\]
be a vector bunde endomorphism into the odd endomorphisms, such that for all $v\in TM$,
\[
c(v)^2 = -\|v\|^2. 
\]
Suppose that $c(g\cdot v) = g\circ c(v) \circ g^{-1}$ for all $g\in G$ and $v\in TM$.

Let $\nabla^{\cE}$ be a $G$-invariant Hermitian connection on $\cE$, such that for all vector fields $v, w \in \cX(M)$,
\[
[\nabla^{\cE}_v, c(w)] = c(\nabla^{TM}_v w),
\]
where $\nabla^{TM}$ is the Levi--Civita connection of the Riemannian metric. Then we have the \emph{Dirac operator}
\[
D := c\circ \nabla^{\cE}\colon \Gamma^{\infty}(\cE) \to \Gamma^{\infty}(\cE).
\]

Let $\psi\colon M\to \kg$ be a $G$-equivariant smooth map, with respect to the adjoint action by $G$ on the Lie algebra $\kg$. This map defines a vector field $v^{\psi}\in \cX(M)$ by
\[
v^{\psi}_m := \ddt \exp(-t\psi(m))\cdot m,
\]
for all $m \in M$. A key assumption we make  is that the set $\Zeroes(v^{\psi})\subset M$ of zeroes of $v^{\psi}$ is \emph{cocompact}, i.e.\ $\Zeroes(v^{\psi})/G$ is compact.
The \emph{Dirac operator deformed by $\psi$} is the operator
\[
D_{\psi} := D - \ii c(v^{\psi}).
\]

Given a real-valued function $\rho \in C^{\infty}(M)^G$, we call a nonnegative function $f\in C^{\infty}(M)^G$ \emph{$\rho$-admissible} if
\[
\frac{f}{\|df\| + f + 1} \geq \rho.
\]
Such functions exist for all $\rho$, see Lemma 3.10 in \cite{HochsSong16a}.

Suppose for now that $G = K$ is compact. A re-interpretation of Theorem 2.9 in \cite{Braverman02} is that there is a function $\rho \in C^{\infty}(M)^k$, such that for all $\rho$-admissible functions $f$, the operator $D_{f\psi}$ is $K$-Fredholm. Then its index
\[
\ind_K(D_{f\psi}) \in KK(C^*K, \C) =\hat R(K) 
\]
is the index studied in \cite{Braverman02} (see Lemma 2.9 in \cite{HochsSong16a}). It is independent of $f$ and $\nabla^{\cE}$.

The main result in \cite{HochsSong16a} is that this generalises to noncompact groups.
\begin{theorem} \label{thm def Dirac G Fred}
There is a function $\rho_G \in C^{\infty}(M)^G$ such that for all $\rho$-admissible functions $f$, the operator $D_{f\psi}$ is $G$-Fredholm for $p$.
\end{theorem}
See Theorem 3.12 in \cite{HochsSong16a}. Furthermore, the $G$-index of $D_{f\psi}$ is independent of $p$, $f$, $\nabla^{\cE}$, and the Riemannian metric on $TM$, as stated precisely in Proposition 3.13 in \cite{HochsSong16a}.

In this paper, we study properties of the $G$-index of deformed Dirac operators, and relations with the analytic assembly map \cite{Connes94}.

\subsection{Dirac induction} \label{sec DInd}

In some of the applications, we will use the Dirac induction isomorphism from the Connes--Kasparov conjecture. We recall the definition of this isomorphism here.
 
Fix a $K$-invariant inner product on the Lie algebra $\kg$, and let $\kp \subset \kg$ be the orthogonal complement to $\kk$.  Suppose that the representation
\[
\Ad\colon K\to \SO(\kp)
\]
lifts to a homomorphism
\beq{eq tilde Ad}
\widetilde{\Ad}\colon K\to \Spin(\kp). 
\eeq
This lift always exists if one replaces $G$ by a double cover. (See the end of this subsection for the case where it does not exist.) Let $S_{\kp}$ be the spinor representation of $\Spin(\kp)$, see e.g.\ Definition 5.11 in \cite{Lawson89}. We view $S_{\kp}$ as a representation of $K$, via the map $\widetilde{\Ad}$. If $\kp$ is even-dimensional, then $S_{\kp}$ has a natural $\Z_2$-grading $S_{\kp} = S_{\kp}^+ \oplus S_{\kp}^-$. By the element $[S_{\kp}] \in R(K)$, we will mean $[S_{\kp}^+] - [S_{\kp}^-]$ in that case.

Existence of the lift \eqref{eq tilde Ad}  is equivalent to $G/K$ having a $G$-equivariant $\Spin$-structure, with spinor bundle $G\times_K S_{\kp} \to G/K$. We will then say that $G/K$ is equivariantly $\Spin$.

For any finite-dimensional representation space $V$ of $K$, we have the $G$-equivariant 
vector bundle
\[
\cE_V := G\times_K(S_{\kp}\otimes V) \to G/K.
\]
Let $\{X_1, \ldots, X_{\dim \kp}\}$ be an orthonormal basis of $\kp$. Consider the Dirac operator
\beq{eq def DV}
D_V := \sum_{j=1}^{\dim \kp} X_j \otimes c(X_j) \otimes 1_V
\eeq
on 
\[
\Gamma^{\infty}(\cE_V) \cong \bigl(C^{\infty}(G) \otimes S_{\kp}\otimes V \bigr)^K.
\]
Here $c$ denotes the Clifford action by $\kp$ on $S_{\kp}$.
It defines an equivariant  $K$-homology class
\beq{eq DV}
[D_V] \in KK_*^G(C_0(G/K), \C).
\eeq

The analytic assembly map from the Baum-Connes conjecture \cite{Connes94, Kasparov83} is a map
\beq{eq ass map}
\mu_Y^G\colon KK_*^G(C_0(Y), \C)\to KK_*(\C, C^*_rG)
\eeq
for any cocompact, proper $G$-space $Y$. Here $C^*_rG$ is now the reduced group $C^*$-algebra of $G$. There is also a version for the maximal group $C^*$-algebra. Applying this map for $Y=G/K$ to the class \eqref{eq DV} yields
\[
\mu_{G/K}^G[D_V] \in KK(\C, C^*_rG).
\]

The \emph{Dirac induction} map
\[
\DInd_K^G\colon R(K) \to KK_*(\C, C^*_rG)
\]
is defined by $[V] \mapsto \mu_{G/K}^G[D_V]$, for finite-dimensional representation spaces $V$ of $K$. It maps into even $KK$-theory if $G/K$ is even-dimensional, and into odd $KK$-theory otherwise. The Connes--Kasparov conjecture, proved by Chabert, Echterhoff and Nest \cite{Chabert03}, states that it is in isomorphism of Abelian groups. (This was proved for linear reductive groups by Wassermann \cite{Wassermann87}.)

By the universal coefficient theorem, we have
\[
KK_*(C^*_rG, \C) \cong \Hom_{\Z}\bigl( KK_*(\C, C^*_rG), \Z \bigr)
\]
via the Kasparov product. Pulling back along Dirac induction therefore defines an isomorphism of Abelian groups
\beq{eq DInd star}
(\DInd_K^G)^*\colon KK(C^*_rG, \C) \xrightarrow{\cong} \Rhat(K).
\eeq

If the lift \eqref{eq tilde Ad} does not exist, i.e.\ $G/K$ is not equivariantly $\Spin$, one still has  a Dirac induction isomorphism. Let $\pi\colon \tilde G \to G$ be a double cover for which \eqref{eq tilde Ad} exists, and let $\tilde K := \pi^{-1}(K)$. Let $u$ be the nontrivial element of $\ker \pi$. Set
\[
R_{\Spin}(K) := \{V \in R(\tilde K); \text{$u$ acts trivially on $V \otimes S_{\kp}$}\}.
\]
Then for $V\in R_{\Spin}(K)$, the tensor product $V \otimes S_{\kp}$ can be viewed as a (virtual) representation of $K$, and the above constructions apply. This yields an isomorphism
\[
\DInd_K^G\colon R_{\Spin}(K) \xrightarrow{\cong} KK_*(\C, C^*_rG).
\]


\section{Induction from slices} \label{sec induction}

We have seen in Theorem \ref{thm def Dirac G Fred} 
that deformed Dirac operators have well-defined $G$-indices. In the rest of this paper, we discuss properties of these indices. One useful tool is the induction result we prove in this section, Theorem \ref{thm induction}. It is a relation between the $G$-index of a deformed Dirac operator on $M$, and the $K$-index of an operator on a slice $N\subset M$. The latter index can be described explicitly in terms of the $L^2$-kernel of the operator.

We keep using the notation and assumptions of Subsection \ref{sec def Dirac}. We fix a smooth, equivariant map $p\colon M\to G/K$, and consider the corresponding slice $N = p^{-1}(eK)$. 


\subsection{The induction result} \label{sec ind thm}

Consider the restriction $\nabla^{\cE|_N}$ of $\nabla^{\cE}$ to $N$. Since $TN \subset TM|_N$, we have the Clifford action
\[
c_N\colon TN\otimes \cE|_N \to \cE|_N.
\]
These combine to the Dirac-type operator
\beq{eq def DN}
D^N\colon \Gamma^{\infty}(\cE|_N) \xrightarrow{\nabla^{\cE|_N}} \Gamma^{\infty}(TN\otimes \cE|_N)\xrightarrow{c_N}\Gamma^{\infty}(\cE|_N).
\eeq

As before, we fix a $K$-invariant inner product on $\kg$.
Let $\kp \subset \kg$ be the orthogonal complement to $\kk$. Then
\beq{eq g k p}
\kg = \kk \oplus \kp.
\eeq
We will identify
\beq{eq TMN}
TN \oplus N\times \kp \cong TM|_N
\eeq
via the map
\[
(v, (n, X))\mapsto v + X^M_n
\]
for $n\in N$, $v\in T_nN$ and $X\in \kp$.

Let $B$ be the given Riemannian metric on $M$. We will consider two $K$-invariant metrics on $TM|_N$. One is simply the restriction $B|_N$. The other is defined by the properties that the decomposition \eqref{eq TMN} is orthogonal, the metric equals $B|_{TN}$ on $TN$, and is defined by the inner product on $\kg$ on $N\times \kp$. We denote this second metric by $B_{\kp}$.

Choose a $K$-equivariant, isometric vector bundle isomorphism
\[
(TM|_N, B_{\kp}) \to (TM|_N, B|_N),
\]
which is the identity on $TN$.
Via this map, the Clifford action
\[
c|_N\colon (TM|_N, B|_N) \to \End(\cE)
\]
defines the Clifford action
\beq{eq cp}
c_{\kp}\colon (TM|_N, B_{\kp}) \to \End(\cE).
\eeq
We have $c_{\kp}|_{TN} = c|_{TN}$.

For a real-valued function $f\in C^{\infty}(M)^G = C^{\infty}(N)^K$, consider the deformed Dirac operator
\[
D^N_{f\psi|_N} = D^N - \ii c_{\kp}( v^{f\psi}|_N)
\]
on $\Gamma^{\infty}(\cE|_N)$. 
%
Under the additional assumption that $\psi(N)\subset \kk$, this is the Dirac operator on $\cE|_N$ deformed by $f\psi|_N$ as in Subsection \ref{sec def Dirac}, but we do \emph{not} make this assumption. (See Subsection \ref{sec psi N k} for other consequences of that assumption.) The analogue of Theorem \ref{thm def Dirac G Fred} still holds, however.
\begin{proposition} \label{prop DN K Fred}
There is a positive function $\rho_N \in C^{\infty}(N)^K = C^{\infty}(M)^G$ such that if $f$ is $\rho_N$-admissible, then $D^{N}_{f\psi|_N}$ is $K$-Fredholm.
\end{proposition}
This proposition is not a direct consequence of Theorem \ref{thm def Dirac G Fred}, since the vector field $v^{\psi}|_N$ is not induced by a map $N\to \kk$ unless $\psi(N)\subset \kk$. For this reason, and to illustrate a simpler approach that is possible for compact groups, we give a separate proof of Proposition \ref{prop DN K Fred} in Subsection \ref{sec DN K Fred}. This is also a simpler proof of Theorem \ref{thm def Dirac G Fred} in the case where $G$ is compact. Furthermore, in the case of trivial groups, the arguments in Subsection \ref{sec DN K Fred} yield a criterion for Callias-type operators to be Fredholm.

Another way to prove Proposition \ref{prop DN K Fred} would have  been to slightly generalise the proof of Proposition 3.15 in \cite{HochsSong16a}. 
Indeed,  write $\psi = \psi_{\kk} \oplus \psi_{\kp}$ according to the decomposition \eqref{eq g k p}. Then if one adds a sixth term 
\[
A_6 := -\ii \sum_{j=1}^{\dim N} c(e_j)c(\nabla^{TN}_{e_j} f v^{\psi_{\kp}})
\]
(where $\{e_1, \ldots, e_{\dim N}\}$ is a local orthonormal frame for $TN$, and $\nabla^{TN}$ is the Levi--Civita connection for the restricted Riemannian metric on $TN$)
in Lemma 4.6 and Proposition 5.2 in \cite{HochsSong16a}, one obtains the estimates necessary to prove a version of Proposition 3.15 in \cite{HochsSong16a} that implies Proposition \ref{prop DN K Fred}.

Let $\rho_N$ be as in Proposition \ref{prop DN K Fred}, and suppose $f$ is $\rho_N$-admissible. Then by Lemma 2.9 in \cite{HochsSong16a}, we have
\[
\ind_K (D^N_{f\psi|_N}) = \bigl[\ker_{L^2} (D^N_{f\psi|_N})^+ \bigr] -  \bigl[\ker_{L^2} (D^N_{f\psi|_N})^-\bigr] \in KK(C^*K, \C) \cong \Rhat(K). 
\]
The induction result relates this index to the $G$-index $\ind_G(\cE, \psi)$ of $D_{f\psi}$, which makes the latter more concrete and computable.
\begin{theorem}[Induction from slices] \label{thm induction}
The multiplicity of every irreducible representation of $K$ in the $L^2$-kernel of $D^N_{f\psi|_N}$ is finite. Under the identification $KK(C_0(G/K)\rtimes G, \C) = \Rhat(K)$ by Morita equivalence, we have
\[
\ind_G(\cE, \psi) = \bigl[\ker_{L^2} (D^N_{f\psi|_N})^+ \bigr] -  \bigl[\ker_{L^2} (D^N_{f\psi|_N})^-\bigr] \quad \in \Rhat(K).  
\]
\end{theorem}

\subsection{Compact groups and deformed Dirac operators} \label{sec DN K Fred}

We will prove
 a slightly more general statement than Proposition \ref{prop DN K Fred}. 
 Consider the setting of Subsection \ref{sec ind thm}.
 Suppose $\psi(N)\subset \kk$. (What follows will later be applied to the component $\psi_{\kk}$ of $\psi$ in $\kk$.) Let a nonnegative function $f\in C^{\infty}(N)^K$  be given. Let $T\in \End(\cE|_N)^K$ be a fibrewise self-adjoint, odd vector bundle endomorphism, such that
\beq{eq c T}
c(v^{\psi})T + Tc(v^{\psi}) = 0,
\eeq
and 
\beq{eq DNfT}
D^N fT + fT D^N \in \End(\cE|_N).
\eeq
Suppose that the pointwise norm of the endomorphism \eqref{eq DNfT} is bounded above by 
\[
\Theta\cdot (\|df\| + f),
\]
for a function $\Theta \in C^{\infty}(N)^K$ (independent of $f$).
Also suppose that the endomorphism $\|v^{\psi}\|^2 + T^2$ of $\cE|_N$ is invertible outside a compact set. For $f\in C^{\infty}(N)^K$, we consider the operator
\[
D^T_{f\psi} = D^N_{f\psi} + fT = D^N +f(-\ii c(v^{\psi}) + T).
\]
on $\Gamma^{\infty}(\cE|_N)$.
This is a combination of a deformed Dirac operator as studied in this paper, and a Callias-type operator \cite{Anghel93, Bruening92b, Bunke95, Callias78, Kucerovsky01}.
\begin{proposition} \label{prop DT K Fred}
There is a positive function $\rho_N \in C^{\infty}(N)^K$ such that if $f$ is $\rho_N$-admissible, then the operator $D^T_{f\psi}$ is $K$-Fredholm.
\end{proposition}

\begin{remark}\label{remark callias}
If $K$ is the trivial group, so that  $\psi = 0$, and if $T^2$ is invertible outside a compact set, then Proposition \ref{prop DT K Fred} shows that the Callias-type operator 
\[
D^T_f = D^N + f \cdot T
\]
is Fredholm on the noncompact manifold $N$, for admissible functions $f$. 
\end{remark}

To deduce Proposition \ref{prop DN K Fred} from Proposition \ref{prop DT K Fred}, write $\psi = \psi_{\kk} \oplus \psi_{\kp}$ according to the decomposition \eqref{eq g k p}. Then 
if we replace $\psi$ by $\psi_{\kk}$ in Proposition \ref{prop DT K Fred}, and set $T:= \ii c_{\kp}(v^{\psi_{\kp}})$, then the conditions \eqref{eq c T} and \eqref{eq DNfT} hold, so Proposition \ref{prop DN K Fred} follows. The least trivial conditions to check are the following.
\begin{lemma}\label{lem DT DN}
The operators $c(v^{\psi_{\kk}})$ and $c_{\kp}(v^{\psi_{\kp}})$ anticommute, and 
\[
D^N fc_{\kp}(v^{\psi_{\kp}})+ fc_{\kp}(v^{\psi_{\kp}}) D^N
\]
is a vector bundle endomorphism of $\cE|_N$.
\end{lemma}
\begin{proof}
The claim follows from the fact that for the metric $B_{\kp}$ on $TM|_N$, tangent vectors to $N$ are orthogonal to tangent vectors defined by elements of $\kp$. This immediately implies that $c(v^{\psi_{\kk}})$ and $c_{\kp}(v^{\psi_{\kp}})$ anticommute. It also implies that, in terms of a local orthonormal frame $\{e_1, \ldots, e_{\dim N}\}$ of $TN$,
\[
D^N fc_{\kp}(v^{\psi_{\kp}})+ fc_{\kp}(v^{\psi_{\kp}}) D^N = \sum_{j=1}^{\dim N} c\bigl(\nabla^{TM}_{e_j} fv^{\psi_{\kp}}\bigr),
\]
so the claim follows.
\end{proof}

Let $\kt \subset \kk$ be a maximal torus. Fix a set of positive roots for $(\kk_{\C}, \kt_{\C})$. Let $\rho_K$ be half the sum of these positive roots (not to be confused with the function $\rho$ as in Theorem \ref{thm def Dirac G Fred}). Let $\Lambda_+ \subset \ii \kt^*$ be the set of dominant integral weights. For $\lambda \in \Lambda_+$, let $V_{\lambda}$ be the irreducible representation space of $K$ with highest weight $\lambda$. 
For $\lambda \in \Lambda_+$, let $L^2(\cE|_N)_{\lambda}$ be the $V_{\lambda}$-isotypical component of $L^2(\cE|_N)$.
\begin{lemma}\label{callias} 
There is a real-valued function $\rho_N \in C^{\infty}(N)^K$, such that if $f$ is $\rho_N$-admissible, then for all $\lambda \in \Lambda_+$, and all $a_{\lambda} > 0$,
the operator
\[
\bigl( (D^T_{f\psi})^2  + a_{\lambda} \bigr)^{-1}\big|_{L^2(\cE|_N)_{\lambda}}
\]
on $L^2(\cE|_N)_{\lambda}$ is compact.
\end{lemma}
\begin{proof}
Analogously to Lemma 4.5 in \cite{HochsSong16a}, we have the local expression, with respect to a local orthonormal frame $\{e_1, \ldots, e_{\dim N}\}$ of $TN$,
\begin{multline*}
(D^T_{f\psi})^2 = \\
(D^N)^2 + f^2(\|v^{\psi}\|^2+T^2) + \ii \sum_{j=1}^{\dim N} c(e_j)c(\nabla^{TM}_{e_j} fv^{\psi}) -2\ii f\nabla^{\cE|_N}_{v^{\psi}} + (D^NfT + fTD^N).
\end{multline*}
By assumption on $T$, there is a function $\Theta_1 \in C^{\infty}(N)^K$ (independent of $f$ and $\lambda$) such that we have the pointwise estimate
\[
 \Bigl\| \ii \sum_{j=1}^{\dim N} c(e_j)c(\nabla^{TM}_{e_j} fv^{\psi}) + (D^NfT + fTD^N)\Bigr\| \leq \Theta_1 (\|df\|+ f)
\]

 Let $C_{\lambda}  > 0$ be such that for all $X\in \kk$, the operator on $V_{\lambda}$ defined by $X$ has norm at most $C_{\lambda}\|X\|$. For every $n\in N$, the operator
\[
\nabla^{\cE|_N}_{v^{\psi}_n} - \calL_{\psi(n)}
\]
is a linear endomorphism of $\cE_n$. Since, for such $n$, we have
\[
\calL_{\psi(n)}|_{L^2(\cE)_{\lambda}} \leq C_{\lambda}\|\psi(n)\|,
\]
there is a positive function $\Theta_2 \in C^{\infty}(N)^K$ (independent of $f$ and $\lambda$)  such that
\[
\Bigl| \nabla^{\cE|_N}_{v^{\psi}} \Bigr| \leq \Theta_2 + C_{\lambda}\|\psi\|,
\]
where use the absolute value of operators as before.
 
We conclude that
\beq{eq est DT}
(D^T_{f\psi})^2|_{L^2(\cE|_N)_{\lambda}}  \geq (D^N)^2|_{L^2(\cE|_N)_{\lambda}} + f^2(\|v^{\psi}\|^2+T^2) - (\Theta_1 + \Theta_2 + C_{\lambda}\|\psi\|)(\|df\| + f)
\eeq
Let $\tilde \rho_N \in C^{\infty}(N)^K$ be a real-valued function such that $(\Theta_1 + \Theta_2)(\tilde \rho_N - 1)$ tends to infinity faster than $\|\psi\|$ as its argument tends to infinity. Choose $\rho_N \in C^{\infty}(N)^K$ such that, outside a relatively compact neighbourhood of the points $n\in N$ where $\|v^{\psi}_n\|^2 + \|T_n\|^2  = 0$, we have
\[
\rho_N \geq \frac{(\Theta_1 + \Theta_2)\tilde \rho_N}{\|v^{\psi}\|^2 + \|T\|^2}.
\]
In addition, choose $\rho_N$ so that it is at least equal to $1$ outside a compact set.
Suppose $f$ is $\rho_N$-admissible. Then by \eqref{eq est DT}, we find that
\[
(D^T_{f\psi})^2|_{L^2(\cE|_N)_{\lambda}}  \geq (D^N)^2|_{L^2(\cE|_N)_{\lambda}} + \zeta_{\lambda},
\]
where $\zeta_{\lambda} \in C^{\infty}(N)^K$ satisfies 
\[
\zeta_{\lambda} \geq \bigl( (\Theta_1 + \Theta_2)(\tilde \rho_N - 1) - C_{\lambda}\|\psi\|\bigr)(\|df\| + f),
\]
outside a relatively compact neighbourhood of the set $\{ n\in N; \|v^{\psi}_n\|^2 + \|T_n\|^2  = 0\}$. Since $f\geq 1$ outside a compact set, and by the assumption on $\tilde \rho_N$, the function on the right hand side tends to infinity as its argument tends to infinity. This implies that $(D^T_{f\psi})^2|_{L^2(\cE|_N)_{\lambda}}$ has discrete spectrum. It follows that
\[
\bigl( (D^T_{f\psi})^2  + a_{\lambda} \bigr)^{-1}
\]
is indeed a compact operator on $L^2(\cE|_N)_{\lambda}$, for all $a_{\lambda} > 0$.
\end{proof}

In the proof of Proposition \ref{prop DT K Fred}, we will use the Casimir element $\Omega_K$ in the centre of the universal enveloping algebra of $\kk$. 
For $\lambda \in \Lambda_+$, 
the element $\Omega_K$ acts on $V_{\lambda}$ as the scalar
\beq{eq Omega lambda}
\|\lambda + \rho_K\| - \|\rho_K\|.
\eeq
(These norms are defined by the same inner product that was used to define $\Omega_K$, the one fixed before to define the metric $B_{\kp}$.)

The operator $(D^T_{f\psi})^2 + 1$ is invertible by Proposition 10.2.11 in \cite{higson00}. Since $\Omega_K$ is a nonnegative operator, the operator $(D^T_{f\psi})^2 + \Omega_K+ 1$ is invertible as well.
The main part of the proof of Proposition \ref{prop DT K Fred} is the following.
\begin{lemma}\label{lem Omega cpt} There is a real-valued function $\rho_N \in C^{\infty}(N)^K$, such that if $f$ is $\rho_N$-admissible, 
the operator
\[
\bigl( (D^T_{f\psi})^2 + \Omega_K + 1 \bigr)^{-1}
\]
on $L^2(\cE|_N)$ is compact.
\end{lemma}
\begin{proof}
For $\lambda \in \Lambda_+$, the Casimir operator $\Omega_K$ acts on $L^2(\cE|_N)_{\lambda}$ as the scalar \eqref{eq Omega lambda}. By $K$-equivariance of $D^T_{f\psi}$, we therefore have the decomposition
\[
\bigl( (D^T_{f\psi})^2 + \Omega_K + 1 \bigr)^{-1} = \bigoplus_{\lambda \in \Lambda_+} \bigl( (D^T_{f\psi})^2|_{L^2(\cE|_N)_{\lambda}} + \|\lambda + \rho_K\| - \|\rho_K\| + 1 \bigr)^{-1}.
\]
By Lemma \ref{callias} every term in this sum is compact. 
Since the norm of a direct sum of operators is the supremum of the norms of the terms, the above direct sum
 converges in the operator norm, to a compact operator. 
\end{proof}

\medskip \noindent \emph{Proof of Proposition \ref{prop DT K Fred}.}
Let $e\in C^{\infty}(K)$. Since $\Omega_K$ commutes with $D^N$, $T$ and $c(v^{\psi})$, we have
\begin{multline}\label{eq diff Omega}
\pi_K(e)\bigl( (D^T_{f\psi})^2 + 1 \bigr)^{-1} - \pi_K(e)\bigl( (D^T_{f\psi})^2 + \Omega_K^r + 1 \bigr)^{-1} \\
= -\bigl( (D^T_{f\psi})^2 + 1 \bigr)^{-1}  \Omega_K \pi_K(e) \bigl( (D^T_{f\psi})^2 + \Omega_K^r + 1 \bigr)^{-1},
\end{multline}
as a simpler analogue of Lemma 4.6 in \cite{HochsSong16a}. Because $\Omega_K \pi_K(e) =\pi_K( \Omega_K e)$ is a bounded operator, Lemma \ref{lem Omega cpt} implies that 
\[
\pi_K(e)\bigl( (D^T_{f\psi})^2 + 1 \bigr)^{-1}
\] 
is a compact operator, which implies the claim.
\hfill $\square$

\subsection{A special case} \label{sec psi N k}

Under some additional assumptions, Theorem \ref{thm induction} takes a simpler form. One of these assumptions is that $\psi(N) \subset \kk$. This assumption is not satisfied in all relevant examples, see Subsection \ref{sec TGK}. In fact, cases where this assumption is \emph{not} satisfied are furthest removed from existing index theory, and therefore potentially the most interesting.

Since we now suppose that $\psi(N) \subset \kk$,
 the operator $D^N_{f\psi|_N}$ in Theorem \ref{thm induction} is precisely of the form studied by Braverman in \cite{Braverman02}. 
Therefore, Braverman's cobordism invariance result, Theorem 3.7 in \cite{Braverman02}, and all of its consequences, generalise to the $G$-index of deformed Dirac operators, under this additional assumption.


Suppose, furthermore,  that $G/K$ is even-dimensional, and  equivariantly $\Spin$, with $\Z_2$-graded spinor bundle
\[
\cE_{G/K} = G \times_K S_\p. 
\]
Let $\Cl(N\times \kp) = N\times \Cl(\kp)$ be the Clifford bundle of $N\times \kp\to N$.
Consider the $\Cl(N\times \kp)$-module $N\times S_{\kp}$, and the Clifford module 
\beq{eq EN}
\cE_N := \Hom_{\Cl(N\times \kp)} (N\times S_{\kp}, \cE).
\eeq
Then, since $S_{\kp}$ is an irreducible representation of $\Cl(\kp)$, 
\beq{eq decomp EN}
\cE = \cE_N \otimes S_{\kp}. 
\eeq
Let $D^{\cE_N}$ be the Dirac operator associated to any $K$-invariant Clifford connection on $\cE_N$. 
Since now the vector bundle endomorphism $c_{\kp}( v^{f\psi}|_N)$ of $\cE|_N$ acts trivially on the factor $N\times S_{\kp}$, we have the deformed Dirac operator
\[
D^{\cE_N}_{f\psi} := D^{\cE_N} - \ii c_{\kp}( v^{f\psi}|_N)
\]
on $\Gamma^{\infty}(\cE_N)$. In terms of the decomposition \eqref{eq decomp EN}, we then have
\[
D^N_{f\psi|_N} = D^{\cE_N}_{f\psi} \otimes 1_{S_{\kp}}.
\]
Therefore, Theorem \ref{thm induction} reduces to the following statement.
\begin{corollary} \label{cor induction psi N k}
In the setting of Theorem \ref{thm induction}, suppose that $\psi(N)\subset \kk$, and that $G/K$ is even-dimensional and equivariantly $\Spin$. Then
under the identification $KK(C_0(G/K)\rtimes G, \C) = \Rhat(K)$ by Morita equivalence, we have
\[
\ind_G(\cE, \psi) =\bigl( \bigl[\ker_{L^2} (D^{\cE_N}_{f\psi})^+  \bigr] -  \bigl[\ker_{L^2} (D^{\cE_N}_{f\psi})^-\bigr] \bigr) \otimes S_{\kp} \quad \in \Rhat(K).  
\]
\end{corollary}

If $D$ is a $\Spinc$-Dirac operator, we will see in Subsection \ref{sec [Q,R]=0} that the $G$-index of its deformation satisfies the \emph{quantisation commutes with reduction} principle in the setting of this subsection.

\begin{remark} \label{rem Sp zero}
If $G$ is semisimple, $G/K$ is even-dimensional, and $\rank(G) \not= \rank(K)$, then the element $[S_{\kp}] \in R(K)$ is zero, see (1.2.5) in \cite{Barbasch83}. (The arguments there actually imply that the same is true for reductive groups.) In this case, Corollary \ref{cor induction psi N k} is a vanishing result for the $G$-index, under the condition that $\psi(N)\subset \kk$. This is an exceptional situation, however, which shows how restrictive this condition is. In Section \ref{sec properties}, we will see many examples of nonzero $G$-indices.
\end{remark}

We have seen that the situations where the $G$-index of deformed Dirac operators has the potential to yield most information not accessible via existing index theory are those where
\begin{enumerate}
\item $G$ is noncompact;
\item $M/G$ is noncompact; and
\item $\psi(N)$ is not contained in $\kk$ (for any choice of $p\colon M\to G/K$).
\end{enumerate}
In Subsection \ref{sec TGK}, we will see a natural class of examples in this new setting. 

In the rest of this section, we prove Proposition \ref{prop DN K Fred} and 
Theorem \ref{thm induction}.

\subsection{An explicit form of the Morita equivalence isomorphism}

Let us define the module $\cM$, that implements the Morita equivalence $C_0(G/K) \rtimes G \sim C^*K$, as in Situation 10 in \cite{Rieffel82}. As a Hilbert $C^*K$-module, it is the completion of $C_c(G)$ in the $C^*K$-valued inner product given by
\[
(f, f')_{C^*K}(k) = \int_{G}\overline{f({g^{-1}})} f'(g^{-1}k)\, dg,
\]
for $f, f' \in C_c(G)$ and $k\in K$. The right action by $C^*K$ is given by
\[
(f\psi)(g) = \int_K f(kg) \psi(k) \, dk,
\]
for $f \in C_c(G)$, $\psi \in C(K)$ and $g\in G$. The representation $\pi_{\cM}$ is given by
\[
(\pi_{\cM}(\varphi)f)(g) = \int_G \varphi(g', gK)f(g'^{-1}g) \delta_G(g')^{1/2}\, dg',
\]
for $\varphi \in C_c(G, C_0(G,K))$, $f\in C_c(G)$ and $g\in G$. (We will always identify maps from $G$ to $C_0(G/K)$ with functions on $G\times G/K$.)

Consider the class
\[
[\cM] := [\cM, 0, \pi_{\cM}] \in KK(C_0(G/K)\rtimes G, C^*K)
\]
defined by $\cM$.
Let $\calH$ be a Hilbert space, and $\pi_K\colon K\to \U(\calH)$ a unitary representation. We will also use the symbol $\pi_K$ for the corresponding representation
\[
\pi_K\colon C^*K \to \cB(\calH).
\]
Let $F$ be a $K$-equivariant, bounded operator on $\calH$.

Consider the representation
\[
\pi_{C_0(G/K)\rtimes G}\colon C_0(G/K)\rtimes G \to \cB\bigl( (L^2(G) \otimes \calH)^K\bigr), 
\]
defined by
\beq{eq def pi GKG}
(\pi_{C_0(G/K)\rtimes G}(\varphi)\sigma)(g) = \int_G \varphi(g, g'K)\delta_G(g')^{1/2} \sigma(g'^{-1}g)\, dg',
\eeq
for $\varphi \in C_c(G, C_0(G/K))$, $\sigma \in  (L^2(G) \otimes \calH)^K$, and $g\in G$. Here $\delta_G$ is the modular function on $G$.

\begin{lemma} \label{lem ME KK}
There is a unitary isomorphism 
\[
\Psi\colon \cM \otimes_{C^*K} \calH \to (L^2(G) \otimes \calH)^K
\]
that intertwines the representation $\pi_{C_0(G/K)\rtimes G}$ and the representation
\[
\pi_{\cM}\otimes 1_{\calH}\colon C_0(G/K) \rtimes G \to \cB\bigl( \cM \otimes_{C^*K} \calH  \bigr), 
\]
and satisfies
\[
\Psi \circ (1_{\cM} \otimes F) = (1_{L^2(G)} \otimes F) \circ \Psi. 
\]
\end{lemma}
\begin{proof}
Consider the map
\[
\Psi\colon C_c(G) \otimes_{\C} \calH \to (C_c(G) \otimes \calH)^K
\]
given by averaging over $K$:
\[
\Psi(f \otimes \xi) (g)= \int_K f(gk)\pi_K(k) \xi\, dk, 
\]
for $f \in C_c(G)$, $\xi \in \calH$ and $g \in G$. One checks that for all $\psi \in C(K) \subset C^*K$, and $f$ and $\xi$ as above,
\[
\Psi(f\psi \otimes \xi) = \Psi(f \otimes \pi_K(\psi)\xi).
\]
Hence the map $\Psi$ descends to a map
\[
C_c(G) \otimes_{C(K)} \calH \to (C_c(G) \otimes \calH)^K,
\]
still denoted by $\Psi$. 
This map has the desired properties.
\end{proof}

Next, suppose that $\calH$ has a $K$-invariant $\Z_2$-grading, $F$ is odd and self-adjoint, and the triple
$
(\calH, F, \pi_K)
$
is a Kasparov $(C^*K, \C)$-cycle. Let
\[
[F] \in KK(C^*K, \C)
\]
be its class.
\begin{proposition} \label{prop KK F}
The triple
\begin{equation} \label{eq KK module 1}
\bigl( (L^2(G) \otimes \calH)^K, 1_{L^2(G)}\otimes F, \pi_{C_0(G/K)\rtimes G} \bigr)
\end{equation}
is a Kasparov $(C_0(G/K) \rtimes G, \C)$-cycle, and its class in $KK(C_0(G/K) \rtimes G, \C)$ equals
\[
[\cM] \otimes_{C^*K} [F].
\]
\end{proposition}
\begin{proof}
Since the operator $F$ commutes with the representation of $C^*K$ in $\calH$, the operator $1_{\cM} \otimes F$ on $\cM \otimes_{C^*K} \calH$ is well-defined. Furthermore, the Kasparov product of $[\cM]$ and $[F]$ is represented by the Kasparov $(C_0(G/K) \rtimes G, \C)$-cycle
\[
\bigl( \cM \otimes_{C^*K} \calH, 1_{\cM}\otimes F, \pi_{\cM} \otimes 1_{\calH} \bigr).
\]
(See e.g. \cite{Blackadar}, Example 18.3.2(a).) This Kasparov cycle is unitarily equivalent to \eqref{eq KK module 1} by Lemma \ref{lem ME KK}.
\end{proof}

\subsection{Product metrics} \label{sec prod metric}

The $K$-invariant metric $B_{\kp}$ on $TM|_N$ extends to a $G$-invariant Riemannian metric on $TM$, which we still denote by $B_{\kp}$. By Proposition 3.13 in \cite{HochsSong16a}, this Riemannian metric leads to the same $G$-index as the original metric, as long as $M$ is complete with respect to $B_{\kp}$.
\begin{lemma}\label{lem M complete}
The manifold $M$ is complete in the metric $B_{\kp}$.
\end{lemma}
\begin{proof}
Note that $G$ is complete in the left invariant Riemannian metric $B_G$ defined by the inner product on $\kg$ used in the definition of $B_{\kp}$. Since $M$ is complete in $B$, the slice $N$ is complete in the metric $B|_{TN}$. Hence $G\times N$ is complete in the product metric $B_{G}\times B|_{TN}$. The quotient $M$ of $G\times N$ is therefore complete in the distance function induced by the Riemannian distance on $G\times N$ (see Proposition 3.1 in \cite{Alekseevsky03}). This equals the distance function defined by the Riemannian metric $B_{\kp}$, so the claim follows.
\end{proof}

Let $L^2(\cE, B_{\kp})$ be the $L^2$-space of sections of $\cE$, defined with respect to the Riemannian density associated to $B_{\kp}$.
We use the metric $B_{\kp}$ for two reasons. The first is that Lemma \ref{lem DT DN} is true for this metric. The second is that it allows us to decompose the space $L^2(\cE, B_{\kp})$ in a way that will allow us to apply Proposition \ref{prop KK F}.

For this decomposition, we 
consider the map
\beq{eq Phi}
\Phi\colon \bigl(C_c(G)\otimes \Gamma_c(\cE|_N)\bigr)^K \to \Gamma_c(\cE)
\eeq
given by
\[
\Phi(\varphi \otimes s)(gn) = \varphi(g) g\cdot s(n),
\]
for $\varphi \in C_c(G)$ and $s\in \Gamma_c(\cE|_N)$ such that $\varphi \otimes s$ is $K$-invariant,  and $g\in G$ and $n\in N$. In general the $K$-invariant simple tensors of the form $\varphi \otimes s$ may not span the whole space $\bigl(C_c(G)\otimes \Gamma_c(\cE|_N)\bigr)^K$. Then we extend $\Psi$ linearly to sums of tensors that are $K$-invariant, while their individual terms may not be. We will tacitly use this convention in the rest of this section.
\begin{lemma} \label{lem Phi unitary}
The map $\Phi$ extends to a $G$-equivariant, unitary isomorphism
\[
\Phi\colon \bigl(L^2(G)\otimes L^2(\cE|_N)\bigr)^K \to L^2(\cE, B_{\kp}).
\]
\end{lemma}
\begin{proof}
Equivariance of $\Phi$ follows directly from the definitions. For surjectivity, note that $\Phi \bigl( \bigl(C_c(G)\otimes \Gamma_c(\cE|_N)\bigr)^K\bigr)$ is dense in $L^2(\cE, B_{\kp})$.
To show that $\Phi$ is an isometry, we consider the $G$-invariant measure $d[g, n]$ on $M=G\times_KN$ induced by the Riemannian density on $G\times N$ associated to the product metric $B_G \times B|_{TN}$ used in the proof of Lemma \ref{lem M complete}. (See e.g.\ \cite{Bourbaki}, Chapter VII, Section 2.2, Proposition 4b.) By a direct verification, the map $\Phi$ is unitary with respect to that measure.
One can show that the measure $d[g,n]$ equals the one given by the Riemannian density associated to $B_{\kp}$. (See Lemma 5.1 in \cite{Hochs14}.) 
\end{proof}

\subsection{Morita equivalence and the $G$-index}

The comments on the Riemannian metric $B_{\kp}$ in Subsection \ref{sec prod metric} allow us to deduce Theorem \ref{thm induction} from Proposition \ref{prop KK F}.

For any $K$-equivariant (real or complex) vector bundle $E\to N$, consider the $G$-equivariant vector bundle 
\[
G\times_K E \to M.
\]
Analogously to \eqref{eq Phi}, we have $G$-equivariant map
\[
\Phi_E\colon \bigl(C^{\infty}(G)\otimes \Gamma^{\infty}(E)\bigr)^K \to \Gamma^{\infty}(G\times_K E),
\]
given by 
\[
\Phi_E(\varphi \otimes s)(gn) = \varphi(g) g\cdot s(n),
\]
for  $\varphi \in C^{\infty}(G)$ and $s\in \Gamma^{\infty}(E)$ such that $\varphi \otimes s$ is $K$-invariant, and $g \in G$ and $n\in N$. If $E = \cE|_N$, this gives
\[
\Phi:= \Phi_{\cE|_N}\colon\bigl(C^{\infty}(G)\otimes \Gamma^{\infty}(\cE|_N)\bigr)^K \to \Gamma^{\infty}(\cE).
\]

Recall that we used the restricted connection $\nabla^{\cE|_N}$ to define the Dirac operator $D^N$ in \eqref{eq def DN}. We will also use a decomposition of the Dirac operator $D$. To define this decomposition, we recall that
we have
\[
M \cong G\times_K N
\]
as in \eqref{eq M G K}. We have a $G$-equivariant isomorphism of vector bundles
\beq{deco tangent}
TM \cong p^*T(G/K) \oplus G\times_K TN.
\eeq
This decomposition of $TM$ yields two projections
\beq{eq pGK pN}
\begin{split}
p_{G/K}\colon&TM \to p^*T(G/K);\\
p_N\colon& TM\to G\times_K TN.
\end{split}
\eeq
Identifying $T^*M\cong TM$ via the Riemannian metric as before, we obtain two partial Dirac operators
\[
\begin{split}
D_{G/K}\colon& \Gamma^{\infty}(\cE) \xrightarrow{\nabla^{\cE}} \Gamma^{\infty}(TM\otimes \cE)\xrightarrow{p_{G/K}\otimes 1_{\cE}} 
 \Gamma^{\infty}(p^*T(G/K)\otimes \cE)
\xrightarrow{c} \Gamma^{\infty}(\cE); \\
D_{N}\colon& \Gamma^{\infty}(\cE) \xrightarrow{\nabla^{\cE}} \Gamma^{\infty}(TM\otimes \cE)\xrightarrow{p_{N}\otimes 1_{\cE}} 
 \Gamma^{\infty}(G\times_K TN \otimes \cE)
\xrightarrow{c} \Gamma^{\infty}(\cE).
\end{split}
\]
Since $p_{G/K} + p_N$ is the identity map on $TM$, we have
\beq{eq decomp D}
D = D_{G/K} + D_{N}.
\eeq
This decomposition played a crucial role in the etimates in \cite{HochsSong16a}.

 The proofs of the following two lemmas are straightforward.
\begin{lemma} \label{lem conn Phi}
The following diagram commutes:
\beq{diag conn Phi}
\xymatrix{
(C^\infty(G)\otimes \Gamma^\infty(\cE|_N))^K \ar[r]^-{\Phi} \ar[dd]^-{1\otimes \nabla^{\cE|_N}}& \Gamma^{\infty}(\cE)\ar[d]^-{\nabla^{\cE}} \\
 & \Gamma^{\infty}(T^*M \otimes \cE)\ar[d]^-{p_N} \\
(C^\infty(G)\otimes \Gamma^\infty(T^*N\otimes \cE|_N))^K \ar[r]_-{\Phi_{T^*N \otimes \cE|_N}}& \Gamma^{\infty}((G\times_K T^*N) \otimes \cE).
}
\eeq
\end{lemma}

\begin{lemma} \label{lem Phi D}
One has
\[
\Phi\circ (1\otimes D^N) = D_N \circ \Phi.
\]
\end{lemma}

Let 
\[
c_{\kp}\colon (TM, B_{\kp}) \to \End(\cE)
\]
be the $G$-equivariant extension of \eqref{eq cp}.
Let $f\in C^{\infty}(M)^G$. Then
\[
\Phi \circ (1\otimes c_{\kp}(fv^{\psi}|_N)) = c_{\kp}(fv^{\psi}) \circ \Phi. 
\]
So Lemma \ref{lem Phi D} implies that
\beq{Phi def Dirac}
\Phi \circ (1\otimes D^N_{f\psi|_N}) = (D_N - \ii c_{\kp}(fv^{\psi}))\circ \Phi.
\eeq
This allows us to prove Theorem \ref{thm induction}

\medskip\noindent
\emph{Proof of Theorem \ref{thm induction}.}
Let $\rho \in C^{\infty}(M)^G$ be as in Theorem \ref{thm def Dirac G Fred}, for the Riemannian metric $B_{\kp}$ and the Clifford action $c_{\kp}$. Let $\rho_N \in C^{\infty}(N)^K = C^{\infty}(M)^G$ be as in Proposition \ref{prop DN K Fred}. Suppose
$f\in C^{\infty}(M)$ is $\max(\rho, \rho_N)$-admissible. For $t\in \R$, consider the operator
\[
D_{f\psi, t} := D_N + tD_{G/K} - \ii f c_{\kp}(v^{\psi}).
\]
Here $D_N$ and $D_{G/K}$ are defined with respect to $B_{\kp}$ and $c_{\kp}$. The arguments in Sections 4 and 5 of \cite{HochsSong16a}, with $D_{G/K}$ replaced by $tD_{G/K}$, show that for all $t>0$, the operator $D_{f\psi, t}$ is $G$-Fredholm for $p$. For $t=0$, this operator is not elliptic. So  we cannot apply the Rellich lemma as in Subsection3.5 of \cite{HochsSong16a} to show that $D_{f\psi, 0}$ is $G$-Fredholm for $p$. However, we saw in Proposition \ref{prop DN K Fred} that the operator $D^N_{f\psi}$ is $K$-Fredholm. Proposition \ref{prop KK F} therefore implies that
\beq{eq 1 times DN}
\Bigl( \bigl(L^2(G) \otimes L^2(\cE|_N)\bigr)^K, 1_{L^2(G)}\otimes \frac{D^N_{f\psi}}{\sqrt{(D^N_{f\psi})^2+1}}, \pi_{C_0(G/K)\rtimes G} \Bigr)
\eeq
is a Kasparov $(C_0(G/K)\rtimes G, \C)$-cycle. 
The isomorphism $\Phi$ 
 intertwines the representations $\pi_{C_0(G/K)\rtimes G}$ and $\pi_{G, G/K}$, up to the factor $\delta_G(g')^{1/2}$ in the definition \eqref{eq def pi GKG} of $\pi_{C_0(G/K)\rtimes}$. On pages 131 and 132 of \cite{Williams07}, it is explained how to remove this factor. 
Then
Lemma \ref{lem Phi unitary} and the equality \eqref{Phi def Dirac} imply that \eqref{eq 1 times DN} is unitarily equivalent to 
\beq{eq D0}
\Bigl(L^2(\cE),  \frac{D_{f\psi, 0}}{\sqrt{(D_{f\psi, 0})^2+1}}, \pi_{G, G/K} \Bigr).
\eeq
The  triple \eqref{eq D0} is therefore also a  Kasparov $(C_0(G/K)\rtimes G, \C)$-cycle, which is to say that  $D_{f\psi, 0}$ is $G$-Fredholm for $p$.

We conclude that for all $t\geq 0$, the operator $D_{f\psi, t}$ is $G$-Fredholm for $p$. So using an operator homotopy, we obtain
\[
\ind_G^p D_{f\psi, 0} = \ind_G^p D_{f\psi, 1} \quad \in KK(C_0(G/K)\rtimes G, \C).
\]
By Proposition 3.13 in \cite{HochsSong16a}, we have
\[
\ind_G(\cE, \psi) = \ind_G^p D_{f\psi, 1}. 
\]
Since the triples \eqref{eq 1 times DN} and \eqref{eq D0} are unitarily equivalent, Proposition \ref{prop KK F} implies that
\[
\ind_G^p D_{f\psi, 0} = [\cM] \otimes_{C^*K} \ind_K (D^N_{f\psi|_N}).
\]
The result now follows from Lemma 2.9 in \cite{HochsSong16a}.
\hfill $\square$


\section{Properties of the $G$-index} \label{sec properties}

The $G$-index of deformed Dirac operators turns out to have several interesting properties. We already saw that it equals Braverman's index if $G$ is compact (see Lemma Lemma 2.9 in \cite{HochsSong16a}). If $M/G$ is compact, we describe how it is related to the analytic assembly map and to an index used by Mathai and Zhang. We work out examples where $M = T^*(G/K)$ and where $G=\R$. The first of these examples gives rise to a $K$-homological version of the Dirac induction isomorphism described in Subsection \ref{sec DInd}.
 Finally, we
 show that for $\Spinc$-Dirac operators, the index satisfies the \emph{quantisation commutes with reduction} principle that was originally formulated in symplectic geometry.

\subsection{The analytic assembly map and the Mathai--Zhang index} \label{sec cocpt}

In this subsection, we suppose that $M/G$ is compact, but $M$ and $G$ may be noncompact. 
Then the $G$-index is closely related to the analytic assembly map and an index defined by Mathai and Zhang \cite{Mathai10}. 

Let
\begin{itemize}
\item $C^*G$ now denote the \emph{maximal} group $C^*$-algebra of $G$;
\item $\mu_M^G\colon KK^G(C_0(M), \C)\to KK(\C, C^*G)$ be the analytic assembly map \eqref{eq ass map};
\item $[1_K] \in KK(\C, C_0(G/K)\rtimes G) \cong R(K)$ be the class corresponding to the trivial representation of $K$;
\item $[1_{G}]  \in KK(C^*G, \C)$ be the class corresponding to the trivial representation of $G$, equal to the class of the  $*$-homomorphism $I^G\colon C^*G \to \C$ given on $L^1(G)$ by integrating functions over $G$.
\end{itemize}
The Mathai--Zhang index was defined in  Definition 2.4 in \cite{Mathai10}, for Dirac operators. It is a numerical index, which is defined in terms of  the $G$-invariant part of the kernel of an operator.

As before, let $p\colon M\to G/K$ be a smooth, equivariant map. Since $M/G$ is compact, the map $p$ is proper. So it induces $p^*\colon C_0(G/K)\to C_0(M)$, and hence $(p_G)^*\colon C_0(G/K)\rtimes G \to C_0(M)\rtimes G$. 
Let
\[
j^G\colon KK^G(C_0(M), \C) \to KK(C_0(M)\rtimes G, C^*G)
\]
be  the descent map (\cite{Kasparov88}, Section 3.11). 
\begin{proposition} \label{prop indices}
If $M/G$ is compact, then there are maps $\ind_G$ and $\ind_{\MZ}$, which on $K$-homology classes defined by elliptic operators are given by the $G$-index and the Mathai--Zhang index of these operators, respectively, such that the following diagram commutes:
\beq{eq diag indices 1}
\xymatrix{
KK^G(C_0(M), \C) \ar[rr]^-{(p_G)_* \circ j^G} \ar@<3pt>@/^3pc/[rrrr]^-{\mu_M^G} \ar@<-3pt>@/_4pc/[rrrrd]_-(0.8){\ind_{MZ}} \ar@<-0pt>@/_0pc/[rrd]_-(0.5){\ind_G} & &
KK(C_0(G/K)\rtimes G,  C^*G) \ar[rr]^-{[1_K] \otimes_{C_0(G/K)\rtimes G} \relbar } \ar[d]^-{\relbar  \otimes_{C^*G}[1_G]} & &
 KK(\C, C^*G) \ar[d]^-{\relbar  \otimes_{C^*G} [1_G]} \\
 & &
KK(C_0(G/K)\rtimes G, \C) \ar[rr]^-{[1_K] \otimes_{C_0(G/K)\rtimes G} \relbar}  & &
 KK(\C, \C).
 \\
}
%
%
%
\eeq
\end{proposition}
This proposition implies that the assembly map and the $G$-index of an elliptic operator $D$ can both be recovered from the class
\begin{equation} \label{fundamental index}
(p_G)_* \circ j^G[D] \in KK(C_0(G/K)\rtimes G, C^*G).
\end{equation}
Furthermore, the Mathai--Zhang index can be recovered from either of these two indices, via the Kasparov product with $[1_G]$ and $[1_K]$, respectively.

To prove Proposition \ref{prop indices}, we consider an odd, self-adjoint, elliptic, $G$-equivariant differential operator on $\cE\to M$.
Then we have the class 
\[
[D]_{C_0(M)\rtimes G} = \Bigl[L^2(\cE), \frac{D}{\sqrt{D^2 + 1}}, \pi_{G, C_0(M)}\Bigr] \in KK(C_0(M)\rtimes G, \C)
\]
as in (the elliptic case of) Proposition 6.4 in \cite{Kasparov14}. Here $\pi_{G, C_0(M)}\colon C_0(M)\rtimes G \to \cB(L^2(\cE))$ is induced by the representations of $C_0(M)$ and $G$ in $L^2(\cE)$.

Consider the class $[D] \in KK^G(C_0(M), \C)$ defined by $D$.
\begin{lemma} \label{lem Khom class M}
We have
\[
[D]_{C_0(M)\rtimes G} = j^G[D] \otimes_{C^*G}[1_G].
\]
\end{lemma}
\begin{proof}
One can check that
\[
j^G[D] = \Bigl[L^2(\cE)\otimes C^*G,  \frac{D_{f\psi}}{\sqrt{D_{f\psi}^2 +1}} \otimes 1, \pi \Bigr] \quad \in KK(C_0(M)\rtimes G, C^*G),
\]
where $\pi\colon C_0(M)\rtimes G \to \cB(L^2(\cE)\otimes C^*G)$ is given by
\[
\pi(\varphi \otimes f)( s\otimes \psi) (g)= \int_G \varphi(g')\psi(g'^{-1}g) f g'\cdot s \, dg',
\]
for $\varphi, \psi \in C_c(G)$, $f\in C_0(M)$, $s\in L^2(\cE)$ and $g\in G$.


The product with $[1_G]$ is just the map functorially induced by the map $I^G$.
So taking this product, we obtain
\[
\begin{split}
j^G[D] \otimes_{C^*G} [1_G]&= \Bigl[L^2(\cE)\otimes C^*G \otimes_{I_G}\C,   \frac{D_{f\psi}}{\sqrt{D_{f\psi}^2 +1}}  \otimes 1\otimes 1, \pi\otimes 1 \Bigr].
\end{split}
\]
Now we have the unitary isomorphism 
\[
\Phi\colon L^2(\cE)\otimes C^*G \otimes_{I_G}\C \xrightarrow{ \cong} L^2(\cE)
\]
given by
\[
\Phi(s\otimes \psi \otimes z) = zI^G(\psi)s,
\]
for $s \in L^2(\cE)$, $\psi \in C_c(G)$ and $z\in \C$. By a direct computation, we find that for all $\varphi, \psi \in C_c(G)$, $f\in C_0(M)$, $s \in L^2(\cE)$  and $z\in \C$,
\[
\Phi \circ (\pi\otimes 1)(\varphi\otimes f) (s\otimes \psi \otimes z) = \pi_{G, M}(\varphi \otimes f) \Phi (s\otimes \psi \otimes z).
\]
\end{proof}

\begin{remark}
Proposition 6.4 in \cite{Kasparov14} states that the class $[D]_{C_0(M)\rtimes G}$ is well-defined even if $D$ is just transversally elliptic. If $D$ is elliptic, then it also defines a class in $K^G(C_0(M), \C)$, and
Lemma \ref{lem Khom class M} is a relation between these classes.
\end{remark}

Define the map $\ind_G$ by commutativity of the following diagram:
\[
\xymatrix{
KK^G(C_0(M), \C) \ar[d]_-{j^G} \ar@<30pt>@/^3pc/[ddd]^-{\ind_G}\\
KK(C_0(M)\rtimes G, C^*G) \ar[d]_-{\relbar \otimes_{C^*G}[1_G]}\\
 KK(C_0(M)\rtimes G, \C) \ar[d]_-{(p_G)_*} \\
KK(C_0(G/K)\rtimes G, \C)
}
\]
Then Lemma \ref{lem Khom class M} shows that this map indeed equals the $G$-index on classes defined by elliptic operators.

Define the map 
\[
\ind_{\MZ}\colon KK^G(C_0(M), \C) \to KK(\C, \C) = \Z
\]
as the composition of the analytic assembly map and the Kasparov product over $C^*G$ with $[1_G]$. 
Bunke showed in the appendix to \cite{Mathai10} that this corresponds to the Mathai--Zhang index on classes of Dirac operators.

\medskip \noindent \emph{Proof of Proposition \ref{prop indices}.}
Consider the diagram
\begin{equation} \label{eq diag indices}
\hspace{-0.8cm}
\xymatrix{
KK^G(C_0(M), \C) \ar[r]^-{j^G} \ar@<0pt>@/^3pc/[rrrr]^-{\mu_M^G} \ar@<-5pt>@/_6pc/[rrrrd]_-(0.8){\ind_{MZ}} \ar@<-0pt>@/_4pc/[rrd]_-(0.8){\ind_G} &
KK(C_0(M)\rtimes G, C^*G) \ar[r]^-{(p_G)_*} \ar[d]^-{\relbar \otimes_{C^*G} [1_G]} &
KK(C_0(G/K)\rtimes G,  C^*G) \ar[rr]^-{[1_K] \otimes_{C_0(G/K)\rtimes G} \relbar } \ar[d]^-{\relbar  \otimes_{C^*G}[1_G]} & &
 KK(\C, C^*G) \ar[d]^-{\relbar  \otimes_{C^*G} [1_G]} \\
 &
 KK(C_0(M)\rtimes G, \C) \ar[r]^-{(p_G)_*} &
KK(C_0(G/K)\rtimes G, \C) \ar[rr]^-{[1_K] \otimes_{C_0(G/K)\rtimes G} }  & &
 KK(\C, \C).
 \\
}
\end{equation}
The two squares with the map $(p_G)_*$ and the products with $[1_K]$ and $[1_G]$, commute by the basic functoriality and associativity properties of $KK$-theory. 
The class $[1_K] \in KK(\C, C_0(G/K)\rtimes G)$ equals the class defined by a cutoff function on $G/K$. Hence commutativity of the top part of this diagram is one of the standard definitions of the assembly map, combined with the fact that the assembly map behaves naturally with respect to the map $p_*$.
The remaining parts of the diagram commute by the comments made above. We conclude that the whole diagram \eqref{eq diag indices} commutes. 
\hfill $\square$

\begin{remark}
In \cite{Braverman14, Mathai13}, a generalisation of the Mathai--Zhang index of deformed Dirac operators to non-cocompact actions is studied. 
If $G$ is compact, then
\[
[1_K] \otimes_{C_0(G/K)\rtimes G} \ind_G(D_{f\psi}) 
\]
is the index studied in \cite{Braverman14, Mathai13}. Since we saw in Proposition \ref{prop indices} that this is also true if $M/G$, instead of $G$,  is compact, this leads the authors to suspect that it holds in general.
\end{remark}

\begin{example}
If $G=K$ is \emph{compact}, so that $M$ is as well, then the class \eqref{fundamental index} in 
\[
KK(C^*K, C^*K) = \Hom(R(K), R(K))
\]
is taking the tensor product with the usual equivariant index of $D$. The map
\[
I^K_*\colon \Hom(R(K), R(K))  \to R(K)
\]
is now given by applying operators on $R(K)$ to the trivial representation. As expected, applying this to the class \eqref{fundamental index} yields the equivariant index of $D$, i.e.\ its image under the assembly map for compact groups.

The Kasparov product with $[1_K]\in KK(C^*K, \C)$ is the map
\[
\Hom(R(K), R(K)) \to \Hom(R(K), \Z)
\]
given by taking the dimension of the invariant part after applying an operator on $R(K)$.  Applying this map to \eqref{fundamental index} yields the map $R(K) \to \Z$ that maps the class of $V \in \hat K$ to
\[
\dim (\ind_K(D) \otimes V)^K = [\ind_K(D):V^*] = [\ind_K(D):V]. 
\]
(Here $\ind_K$ now denotes the usual equivariant index.)
In other words, we recover the fact that the $K$-index of $D$ as defined in \cite{HochsSong16a} 
is the image  in $\Rhat(K)$ of the usual $K$-index of $D$.
\end{example}

\subsection{Another relation with the assembly map}

As in Subsection \ref{sec cocpt}, assume that $M/G$ is compact. In addition, we assume  that $G/K$ is even-dimensional and equivariantly $\Spin$. Consider the Dirac induction isomorphism
\[
\DInd_K^G\colon R(K)\xrightarrow{\cong} KK(\C, C^*_rG)
\]
described in Subsection \ref{sec DInd}.
We now use the analytic assembly map with respect to the \emph{reduced} group $C^*$-algebra $C^*_rG$, which we still denote by $\mu_M^G$. 

In this setting, we can express the $G$-index in terms of the assembly map.
\begin{proposition} \label{prop ass Gind}
Under the identification $KK(C_0(G/K)\rtimes G, \C) = \Rhat(K)$ via Morita equivalence, we have
\beq{eq ass Gind}
\ind_G(D_M) = (\DInd_K^G)^{-1}\bigl(\mu_M^G[D_M] \bigr) \otimes S_{\kp} \quad \in R(K) \hookrightarrow \Rhat(K).
\eeq
\end{proposition}
\begin{proof}
Since we may take $\psi = 0$ in the cocompact case, the condition in Subsection \ref{sec psi N k} that $\psi(N)\subset \kk$ is automatically satisfied. By Corollary \ref{cor induction psi N k}, we therefore have
\beq{eq ind DM DN}
 \ind_G(D_M) =  [\cM] \otimes_{C^*K} \bigl(  \ind_K (D^{\cE_N})\otimes S_{\kp}\bigr).
\eeq
Hence the claim follows from the fact that
\[
\mu_M^G[D] = \DInd_K^G \bigl(\ind_K(D^{\cE_N})\bigr).
\]
This was proved for $\Spinc$-Dirac operators in Theorem 4.5 in \cite{Hochs09} and Theorem 4.8 in \cite{Mathai14}. The arguments apply to general Dirac operators, however.
\end{proof}

Proposition \ref{prop ass Gind} implies that the $G$-index of Dirac operators is determined by the assembly map. Conversely, the index $\mu_M^G[D_M]$ can be expressed in terms of $\ind_G(D_M)$ precisely if tensoring with $S_{\kp}$ is an injective operation on $R(K)$.
\begin{lemma}
Suppose that $G$ is semisimple with discrete series, i.e.\ $\rank(G) = \rank(K)$. Then the map from $R(K)$ to $R(K)$ given by the tensor product with $S_{\kp}$ is injective.
\end{lemma}
\begin{proof}
Let $T^{\reg}$ be the set of regular elements in a maximal torus $T$ of $K$.
Since $\rank(G) = \rank(K)$,  it was noted in Remark 2.2 in \cite{Parthasarathy72} that the character $\chi_{S_p} = \chi_{S_{\kp}^-} - \chi_{S_{\kp}^-}$ of $S_{\kp}$ satisfies
\[
\chi_{S_{\kp}}|_{T} = \prod_{\alpha \in R_n^+} (e^{\alpha/2} - e^{-\alpha/2}),
\]
where $\alpha$ runs over a set of positive noncompact roots. This function is nonzero on $T^{\reg}$, and hence on  the open dense subset $K \cdot T^{\reg} \subset K$, where $K$ acts on itself by conjugation. Therefore, multiplication by $\chi_{S_{\kp}}$ is injective.
\end{proof}
In the setting of this lemma, we have
\[
\mu_M^G[D_M] = \DInd_K^G[V],
\]
where the character of $[V] \in R(K)$ equals the character of $\ind_G(D_M)$ divided by $\chi_{S_{\kp}}$.

%
%

\begin{example}
Suppose that $M = G/K$. For $V\in \hat K$,  let $D_V$ be the Dirac operator on the Clifford module $G\times_K(S_{\kp} \otimes V)\to G/K$, as defined in \eqref{eq def DV}. Since $M/G$ is compact, this operator is $G$-Fredholm without the need of a deformation term. Since $M$ is now a homogeneous space, all operators on $G$-equivariant vector bundles over $M$ are transversally elliptic, including the zero operator. Therefore, Proposition 6.4 in \cite{Kasparov14} implies that
\[
\ind_G(D_V) = \bigl[ (L^2(G)\otimes S_{\kp} \otimes V)^K, 0, \pi_{G, G/K} \bigr] \quad \in KK(C_0(G/K)\rtimes G, \C),
\]
via a linear operator homotopy. By Proposition 2.10 in \cite{HochsSong16a}, this class corresponds to $S_{\kp} \otimes V \in \Rhat(K)$. 
The right hand side of \eqref{eq ass Gind} for this operator equals $V\otimes S_{\kp}$, so that we obtain an independent verification of Proposition \ref{prop ass Gind} in this case.

One could view the $G$-index of $D_V$ as a $K$-homological analogue of the Dirac induction of $V$. The twist by $S_{\kp}$ one obtains in this way makes this slightly unnatural though, also in view of Remark \ref{rem Sp zero}. In Subsection \ref{sec TGK}, we define a more natural notion of $K$-homological Dirac induction, using a non-cocompact action.
\end{example}

\subsection{$K$-homological Dirac induction} \label{sec TGK}

We assume that $G/K$ is even-dimensional and equivariantly $\Spin$. The Dirac induction isomorphism
\[
\DInd_K^G\colon R(K) \to K_*(C^*_rG)
\]
 involves Dirac operators on $G/K$. Using the $G$-index, we define a version of Dirac induction in terms of Dirac operators on the  non-cocompact manifold 
\[
M = T^*(G/K) \cong G \times_K \p.  
\]
(We identify $\kp \cong \kp^*$ using a fixed $K$-invariant inner product on $\kg$.)
Consider the map $p\colon T^*(G/K)\to G/K$ given by $p([g, X]) = gK$, for $g\in G$ and $X\in \kp$.
As in (\ref{deco tangent}), we have a $G$-equivariant decomposition of the tangent bundle 
\[
TM \cong p^*T(G/K) \oplus G\times_K (T\p).
\]
Consider the $K$-equivariant vector bundle
 $\cE_\p := \p \times S_\p \to \kp$, and the $G$-equivariant vector bundle $\cE_{G/K} := G \times_K S_\p \to G/K$.
Let us form
\[
\cE_M := p^*\cE_{G/K} \otimes \big(G\times_K \cE_\p),
\]
which defines a $G$-equivariant spinor bundle on $M$. Here, and in the remainder of this subsection, we use \emph{graded} tensor products. Note that the bundle $\cE_M$ contains two factors $S_{\kp}$, coming from the spinor bundles on $G/K$ and $\kp$.
 The natural Clifford actions by $T\kp$ on $\cE_\p$ and by $T(G/K)$ on $\cE_{G/K}$ (both denoted by $c$) combine to a 
 Clifford action 
\[
c\colon TM \cong p^*T(G/K) \oplus G\times_K (T\p) \to \End(\cE_M), 
\]
given by
\[
c(v, w) = c(v) \otimes 1 + 1 \otimes c(w),
\]
for $v\in p^*T(G/K)$ and $w\in T\kp$.

Let $V \in \hat K$.
Consider the $G$-equivariant vector bundle
\[
 \cE_V := \cE_M \otimes ( G \times_K (\p \times V)) \to M. 
\]
By  construction, we have
\beq{iso space}
\Gamma^\infty(\cE_V) \cong \bigl( C^\infty(G) \otimes S_\p \otimes V  \otimes C^\infty(\p, S_\p) \bigr)^K. 
\eeq
Let $D_\p$ be the $K$-equivariant Dirac operator on $C^\infty(\p, S_\p)$  so that
\[
[D_\p] \in KK^K(C_0(\p), \C)
\] 
is the \emph{fundamental class}. Let $D_{V}$ be the operator on 
\[
 C^\infty(G) \otimes S_\p  \otimes V
\]
defined as in \eqref{eq def DV}. We take the Dirac operator $D_M^V$ on \eqref{iso space} to be 
\[
D_M^V := D_{V} \otimes 1 + 1 \otimes D_\p
\]

Let $\psi\colon M \to \g$ be the $G$-equivariant map defined by
\[
\psi[g, X] =\Ad(g)X,
\] 
for $g\in G$ and $X\in \kp$. The vanishing set $\Zeroes(v^{\psi})$ is the zero section of the vector bundle $G\times_K \p \to G/K$, and hence cocompact. Let $\rho \in C^{\infty}(M)^G$ be as in 
 Theorem \ref{thm def Dirac G Fred}, and let $f\in C^{\infty}(M)^G$ be $\rho$-admissible.
Consider the deformed Dirac operator 
\[
D_{f\psi}^V := D_M^V - \ii f c(v^\psi).
\]
on $\cE_V$.
It has a $G$-index 
\[
\ind_G(\cE_V, \psi) \in KK(C_0(G/K)\rtimes G, \C). 
\]
\begin{proposition} \label{prop TGK}
Under the isomorphism $KK(C_0(G/K)\rtimes G, \C) \cong \Rhat(K)$ given by Morita equivalence, we have
\[
\ind_G(\cE_V, \psi) = [V].
\]
\end{proposition}

Let us prove this proposition.
We denote by $\overline{S_\p} = \p \times S_\p$ the trivial vector bundle over $\p$ with fibre $S_{\kp}$. Consider the vector bundle endomorphism  $\beta$ of $\overline{S_\p}$ given by
\[
\beta (X, \xi) := (X, c(X)\xi),
\]
for $X\in \kp$ and $\xi \in S_{\kp}$. It is invertible outside the compact set $\{0\}\subset \kp$, and hence defines a class in 
 the  ($K$-equivariant) topological $K$-theory of $\p$.  This is the \emph{Bott element}, denoted by $[\beta]$. 

We choose an orthonormal basis $\{X_1,\dots, X_{\dim \kp}\}$ of $\p$. Let $\{\eta_1, \ldots, \eta_{\dim \kp}\}$ be the dual basis of $\kp^*$. Then the identity map on $\kp$ equals  $\sum_{j=1}^{\dim \kp} X_j \otimes \eta_j$, and we have
\begin{equation}\label{beta}
\beta = \sum_{j=1}^{\dim \kp}  c(X_j) \otimes \eta_j. 
\end{equation}
Consider the operator
\[
D_{f} := D_\p \otimes 1-  1\otimes \ii f \cdot \beta
\]
on $\Gamma^{\infty}(\cE_{\kp} \otimes \overline{S_\p})$, and the bounded operator
\[
F_{f} := \frac{D_{f}}{\sqrt{1+D_{f}^2}}
\]
on $L^2(\p, \cE_\p \otimes \overline{S_\p})$.
\begin{lemma}
\label{index pair}
There is a real-valued function $\rho \in C^{\infty}(\p)^K$, such that if $f$ is $\rho$-admissible, the operator 
$F_f$
is Fredholm, and its index is given by the trivial $K$-representation. 
\end{lemma}
 \begin{proof}
Since $ \overline{S_\p}$ is a trivial vector bundle over $\p$, we have
\[
L^2(\p, \cE_\p \otimes \overline{S_\p}) \cong L^2(\p, \cE_\p)\otimes S_\p.
\]
By (\ref{beta}), the operator $D_f$ can be rewritten as
\[
D_\p \otimes 1 - \ii \sum_{j=1}^{\dim \kp} f\eta_j \otimes c(X_j).
\] 
Therefore, $[D_\p, f\beta] \in \mathrm{End}(\cE_\p \otimes \overline{S_\p})$, so that $D_{f}$ is an operator as in Remark \ref{remark callias}. That remark therefore implies that there is a function $\rho$ such that $F_f$ is Fredholm if $f$ is $\rho$-admissible. Then the index of $F_f$ is the class
\[
[D_f] := \bigl[L^2(\p, \cE_\p \otimes \overline{S_\p}), F_f \bigr] \in KK^K(\C, \C).
\]

To calculate this index, we write
\[
F_{\kp}:= \frac{D_\p}{\sqrt{1+D_\p^2}}
\]
We view $\Gamma_0(\overline{S_\p})$ as a Hilbert $C_0(\p)$-module, and consider the representation of $C_0(\kp)$ in
$L^2(\p, \cE_\p )$ given by pointwise multiplication.
Then
\[
L^2(\p, \cE_\p \otimes \overline{S_\p}) \cong \Gamma_0(\overline{S_\p}) \otimes_{C_0(\p)} L^2(\p, \cE_\p ).
\]
One can check that $F_{f}$ is an $F_{\kp}$-connection \cite[Definition 2.6]{Kasparov88}. 
This implies that the element $[D_f] \in KK^K(\C, \C)$ is the Kasparov product over $C_0(\kp)$ of the classes $[f\beta] \in KK^K(\C, C_0(\kp))$ and $[D_{\kp}] \in KK^K(C_0(\kp), \C)$.
(One can also find a proof in Lemma 3.1 in \cite{Kucerovsky01}.) 

Note that $f$ is a positive function. By a homotopy, we have $[f \cdot \beta] = [\beta]$.  By Proposition 11.4.5 in \cite{higson00}, the Kasparov product of $[D_\p]$ and $[\beta]$ is the trivial representation of $K$.  
\end{proof}

\medskip \noindent\emph{Proof of Proposition \ref{prop TGK}.}
By Theorem \ref{thm induction}, the Morita equivalence isomorphism maps  $\ind_G(\cE_V, \psi)$
to the
$K$-index of the operator
$
D_f \otimes 1_V
$
on 
\[
L^2(\p, \cE_\p \otimes \overline{S_\p} ) \otimes V.
\]  
By Lemma \ref{index pair},  this index 
equals $[V] \in \Rhat(K)$. 
\hfill $\square$

\medskip
Motivated by Proposition \ref{prop TGK}, we define a $K$-homological version of Dirac induction.
\begin{definition}
For any $V \in \hat K$,  the \emph{$K$-homological Dirac induction} of $V$ is the class
\[
\widehat{\DInd}_K^G(V):= \ind_G(\cE_V, \psi) \in KK(C_0(G/K)\rtimes G, \C).
\]
\end{definition}
Proposition \ref{prop TGK} implies that this is an isomorphism.
\begin{theorem}
$K$-homological  Dirac induction  defines an isomorphism
\[
\Rhat(K) \xrightarrow{\cong}  KK(C_0(G/K)\rtimes G, \C).
\]
\end{theorem}

Of course,  the fact that $\Rhat(K) \cong  KK(C_0(G/K)\rtimes G, \C)$ is not new, or as deep as the Dirac induction isomorphism for $K$-theory. But
we found it interesting that this isomorphism can be described in terms of the $G$-index for a non-cocompact action. For example, this implies that any element of $KK(C_0(G/K)\rtimes G, \C)$ can be realised as a $G$-index for such an action. Furthermore, in this class of examples,  
 the image of $N = \kp$ under
 $\psi$ does not lie inside $\kk$. (Nor for any other choice of the map $M\to G/K$.) Examples with these properties are furthest removed from existing index theory.

\begin{remark}
Let $V = \C$ be the trivial representation of $K$. The map $\psi$ is a \emph{moment map} for the standard symplectic form on $T^*(G/K)$. Hence, following \cite{Zhang14, Paradan11}, one can interpret the $G$-index of $D_{f\psi}^{\C}$ as the \emph{geometric quantisation} $Q_G(T^*(G/K))$ of $T^*(G/K)$. By Propositions 2.10 in \cite{HochsSong16a} and Proposition \ref{prop TGK}, one then has
\[
Q_G(T^*(G/K)) = \bigl[ L^2(G/K), 0, \pi_{C_0(G/K)\rtimes G}\bigr] \quad \in KK(C_0(G/K)\rtimes G, \C).
\]
This looks natural, especially since the representation of $G$ in $L^2(G/K)$ is encoded in $\pi_{C_0(G/K)\rtimes G}$. However, one loses much of this information after applying the homotopy relation in $K$-homology.
 \end{remark}

\begin{example}
Consider the situation of this subsection where $G = \R$, but we allow more general equivariant maps $\psi\colon M\to \kg$ than the one used above. Then we obtain an explicit example of how the $G$-index can depend on the map $\psi$. 

In this case, the map $\psi\colon \R^2 \to \R$ is of the form  
\[
\psi(x, y) = \zeta(y)
\]
for a real-valued function $\zeta \in C^{\infty}(\R)$. The $G$-index of the Dirac operator deformed by $\psi$ now lies in
\[
KK( C_0(\R)\rtimes \R, \C) \cong KK(\C, \C) \cong \Z. 
\]
It equals the difference of the dimensions of the spaces of square-integrable functions that are scalar multiples of the functions $s_-$ and $s_+$ on $\R$, respectively, defined by
\[
s_{\pm}(y)=  e^{\pm\int^y_0 \zeta(t)dt},
\]
for $y\in \R$.

Suppose that the function $\zeta$ is nowhere vanishing over $\R$, for example, $\zeta>0$ on $\R$. Then neither $s_-$ nor $s_+$ will be square-integrable. So the $G$-index of the deformed Dirac operator
equals zero. Note that the induced vector field $v^\psi$ is nowhere vanishing over $M$ in this case. 
If $\zeta(y) = \pm y$, then $s_{\mp}$ is square-integrable, but $s_{\pm}$ is not. It follows that the $G$-index  of the deformed Dirac operator is equal to $\pm1$,  with kernel spanned by the function $y\mapsto e^{-\frac{y^2}{2}}$ (in even or odd degree, respectively).
\end{example}

\subsection{$\Spinc$-quantisation commutes with reduction} \label{sec [Q,R]=0}

In this subsection, we consider the special case where $D$ is a $\Spinc$-Dirac operator. Suppose $G$ is reductive. Suppose $M$ is even-dimensional and  has a $G$-equivariant $\Spinc$-structure. Let $\cE$ be the associated spinor bundle. Furthermore,
 that $G/K$ is even-dimensional and equivariantly $\Spin$.
As before, let $S_{\kp} \in R(K)$ be the $\Spin$ representation  of $\p$.
The Clifford module $\cE_N\to N$ as in \eqref{eq EN} is now the spinor bundle of a $K$-equivariant $\Spinc$-structure on $N$. (See Proposition 3.10 in \cite{Mathai14}.)

Let $L\to M$ be the determinant line bundle of the $\Spinc$-structure. Then $L_N := L|_N \to N$ is the determinant line bundle of the $\Spinc$-structure on $N$ whose spinor bundle is $\cE_N$. Let $\nabla^{L_N}$ be a $K$-invariant, Hermitian connection on $L_N$. It defines a \emph{$\Spinc$-moment map}
\[
\psi_N\colon N\to \kk^*
\]
by 
\[
2\sqrt{-1} \langle \psi_N, X\rangle = \calL^{L_N}_X - \nabla^{L_N}_{X_N} \quad \in \End(L_N) = C^{\infty}(N),
\]
for all $X\in \kk$. Here $\calL^{L_N}$ is the Lie derivative of sections of $L_N$. In Section 3.1 of \cite{Hochs09}, a $G$-invariant, Hermitian connection $\nabla^{L}$ on $L$ is constructed, for which the associated $\Spinc$-moment map $\psi\colon M\to \kg^*$ (defined analogously to $\psi_N$), restricts to $\psi_N$ on $N$.

Fix a $K$-invariant inner product $(\relbar, \relbar)^{\kg}$ on $\kg$. We use this to identify $\kk^* \cong \kk$, and hence to view $\psi_N$ as a map into $\kk$. Furthermore, consider the trivial vector bundle $M\times \kg \to M$, on which $G$ act as
\[
g\cdot (m, X) = (gm, \Ad(g)X),
\]
for $g\in G$, $m\in M$ and $X\in \kg$. We have a $G$-invariant metric on this bundle, defined by
\[
(X, Y)_{gn} := (\Ad(g^{-1})X, \Ad(g^{-1})Y)^{\kg},
\]
for $X, Y \in \kg$, $g\in G$ and $n\in N$. Using this to identify $M\times \kg^* \cong M\times \kg$, we view $\psi$ as a map into $\kg$.

Let $\rho, \rho_N \in C^{\infty}(M)^G$ be as in Theorems \ref{thm def Dirac G Fred} and \ref{thm induction}. Suppose $f\in C^{\infty}(M)^G$ is $\max(\rho, \rho_N)$-admissible. Consider the deformed Dirac operator $D_{f\psi}$. We will see that the image of its $G$-index in $\Rhat(K)$ decomposes into irreducible representations according to the quantisation commutes with reduction principle. In the $\Spinc$-setting, this principle was first proved for compact groups and manifolds in \cite{Paradan14}. This was generalised to noncompact manifolds, but still compact groups, in \cite{HochsSong15}.

Let us make this precise.
For any $\lambda, \nu$ in the set $\Lambda_+$ of dominant weights  used before, let $n^{\nu}_{\lambda}$ be the nonnegative integers such that
\[
V_{\lambda} \otimes S_{\kp} = \bigoplus_{\nu\in \Lambda_+} n^{\nu}_{\lambda} V_{\nu}.
\]
For $\lambda \in \Lambda_+$, let 
\[
M_{\lambda} := \psi^{-1}(\lambda/\ii)/G_{\lambda},
\]
be the reduced space at $\lambda$,
where $G_{\lambda}< G$ is the stabiliser of $\lambda$ with respect to the coadjoint action. Since $G$ is reductive, we have
\[
M_{\lambda} = N_{\lambda} := \psi_N^{-1}(\lambda/\ii)/K_{\lambda}.
\]
(See Proposition 3.13 in \cite{Mathai14}.)

Suppose $\psi$ is $G$-proper, in the sense that the inverse image of a cocompact set is cocompact. Then $M_{\lambda}$ is compact.
This space has a $\Spinc$-quantisation
\[
Q(N_{\lambda}) \in \Z,
\]
as defined in Section 5 of \cite{Paradan14}. In the sufficiently regular case, this is the index of a $\Spinc$-Dirac operator with respect to a $\Spinc$-structure induced by the one on $N$. Then this index equals the index of a corresponding $\Spinc$-Dirac operator on $M_{\lambda}$ (see Proposition 3.14 in \cite{Mathai14}). It therefore makes sense to define
\[
Q(M_{\lambda}) := Q(N_{\lambda}).
\]
\begin{corollary}[$\Spinc$-quantisation commutes with reduction]\label{thm [Q,R]=0} 
Suppose that $G$ is reductive, and that $G/K$ is even-dimensional and equivariantly $\Spin$.
Consider the multiplicities $m_{\lambda} \in \Z$ in 
\[
\ind_G(\cE, \psi) = \sum_{\lambda \in \Lambda_+} m_{\lambda}V_{\lambda} \quad \in \Rhat(K)\cong KK(C_0(G/K)\rtimes G, \C) .
\]
If the action by $G$ on $M$ has Abelian stabilisers on an open dense subset of $M$, then for all $\lambda \in \Lambda_+$,
\[
m_{\lambda} = \sum_{\nu\in \Lambda_+} n^{\nu}_{\lambda} Q(M_{\nu + \rho_K}).
\] 
In general, without an assumption on the stabilisers, we have
\[
m_{\lambda} = \sum_{\nu\in \Lambda_+} n^{\nu}_{\lambda} \Bigl( \sum_{j=1}^{k_{\lambda}} Q(M_{\nu + \rho_j}) \Bigr),
\]
for a finite set $\{\rho_1, \ldots, \rho_{k_{\lambda}}\} \subset i\kt^*$, as specified in Theorem 1.4 in \cite{Paradan14}.
\end{corollary}
\begin{proof} We saw that in this case, $\psi(N)\subset \kk$. Therefore, 
Corollary \ref{cor induction psi N k} implies that
\[
\begin{split}
m_{\lambda} &= \dim \bigl(\ind_K D^{\cE_N}_{f\psi} \otimes S_{\kp} \otimes V_{\lambda}^* \bigr)^K \\
 &= \sum_{\nu\in \Lambda_+} n^{\nu}_{\lambda} \dim \bigl(\ind_K D^{\cE_N}_{f\psi} \otimes  V_{\nu}^* \bigr)^K.
\end{split}
\]
Here we used the fact that $S_{\kp}^* = S_{\kp}$. By Theorem 3.9 in \cite{HochsSong15}, the multiplicity
\[
 \dim \bigl(\ind_K D^{\cE_N}_{f\psi} \otimes  V_{\nu}^* \bigr)^K
\]
is given by the desired expression.
\end{proof}

\begin{remark}
We always assumed that $\Zeroes(v^{\psi})/G$ was compact. In \cite{HochsSong15}, however, it was explained how to handle the case where this condition does not hold, if $G$ is compact. Via Corollary \ref{cor induction psi N k}, the same methods apply to noncompact $G$. Therefore, there is still a well-defined index, and Corollary \ref{thm [Q,R]=0} still holds, if $\Zeroes(v^{\psi})/G$ is noncompact. It is still essential that $\psi$ is $G$-proper.
\end{remark}

\appendix

\section{Notation and conventions} \label{app not}

\begin{itemize}
\item If $X$ is a locally compact Hausdorff space, then $C(X)$ is the algebra of continuous, complex valued functions on $X$. We will mention explicitly where we assume a function to be real-valued. We write $C_b(X)$, $C_0(X)$ and $C_c(X)$ for the algebras of bounded functions, functions vanishing at infinity, and compactly supported functions in $C(X)$, respectively. The sup-norm of a function $f \in C_b(X)$ is denoted by $\|f\|_{\infty}$.
\item
If $\cE\to X$ is a vector bundle (tacitly assumed to be continuous), then $\Gamma(\cE)$ is the space of continuous sections of $\cE$, and $\Gamma_c(\cE)\subset \Gamma(\cE)$ is the subspace of compactly supported sections. If $\cE$ is equipped with a metric, then $\Gamma_0(\cE)$ is the space of continuous sections vanishing at infinity.
If, in addition, 
 $X$ is equipped with a Borel measure, then $L^2(\cE)$ is the space of $L^2$-sections of $\cE$. If $\cE'\to X'$ is another vector bundle, then $\cE \boxtimes \cE' \to X\times X'$ is the exterior tensor product of the two.
\item Analogously, if $M$ is a smooth manifold, we have the algebras $C^{\infty}(M)$, $C^{\infty}_b(M)$, $C^{\infty}_0(M)$ and $C^{\infty}_c(M)$ of smooth, complex valued functions on $M$ with the various growth/decay properties towards infinity. The tangent bundle projection of $M$ is denoted by $\tau_M\colon TM\to M$. The space of smooth vector fields on $M$ is denoted by $\cX(M)$. The set of zeroes of a vector field $v\in \cX(M)$ is denoted by $\Zeroes(v)$.
 If $\cE\to M$ is a (smooth) vector bundle, then we have the space $\Gamma^{\infty}(\cE)$ of smooth sections, and its subspace $\Gamma^{\infty}_c(\cE)$ of compactly supported sections. 
\item If $M$ has a Riemannian metric, we will use it to identify $T^*M \cong TM$. We denote the Levi--Civita connection on $TM$ by $\nabla^{TM}$.
If, in addition,  $\cE$ has a metric, the space $L^2(\cE)$ will be defined with respect to the Riemannian density. 
\end{itemize}



\begin{small}

\bibliographystyle{plain}
\bibliography{mybib}

\end{small}

\end{document}